\documentclass{article}%
\usepackage{amsmath}
\usepackage{graphicx}
\usepackage{amsfonts}
\usepackage{amssymb}%
\providecommand{\U}[1]{\protect\rule{.1in}{.1in}}
\newtheorem{theorem}{Theorem}

\newtheorem{corollary}{Corollary}

\newtheorem{lemma}{Lemma}

\newtheorem{remark}{Remark}

\newenvironment{proof}[1][Proof]{\textbf{#1.} }{\ \rule{1em}{1em}}
\begin{document}

\title{Inviscid dynamical structures near Couette flow}
\author{Zhiwu Lin and Chongchun Zeng\\School of Mathematics\\Georgia Institute of Technology\\Atlanta, GA 30332, USA}
\maketitle

\begin{abstract}
Consider inviscid fluids in a channel $\left\{  -1<y<1\right\}  $. For the
Couette flow $\vec{v}_{0}=\left(  y,0\right)  $, the vertical velocity of
solutions to the linearized Euler equation at $\vec{v}_{0}\,$\ decays in time.
At the nonlinear level, such inviscid damping has not been proved. First, we
show that in any (vorticity) $H^{s}\left(  s<\frac{3}{2}\right)
\ $neighborhood of Couette flow, there exist non-parallel steady flows with
arbitrary minimal horizontal period. This implies that nonlinear inviscid
damping is not true in any (vorticity) $H^{s}\left(  s<\frac{3}{2}\right)  $
neighborhood of Couette flow and for any horizontal period. Indeed, the long
time behavior in such neighborhoods are very rich, including nontrivial steady
flows, stable and unstable manifolds of nearby unstable shears. Second, in the
(vorticity) $H^{s}\left(  s>\frac{3}{2}\right)  $ neighborhood of Couette, we
show that there exist no non-parallel steadily travelling flows $\vec
{v}\left(  x-ct,y\right)  $, and no unstable shears. This suggests that the
long time dynamics in $H^{s}\left(  s>\frac{3}{2}\right)  $ neighborhoods of
Couette might be much simpler. Such contrasting dynamics in $H^{s}$ spaces
with the critical power $s=\frac{3}{2}\ $is a truly nonlinear phenomena, since
the linear inviscid damping near Couette is true for any initial vorticity in
$L^{2}.$

\end{abstract}

\section{Introduction}

Consider the incompressible inviscid fluid in a channel $\left\{  \left(
x,y\right)  \ |\ -1\leq y\leq1\right\}  $, satisfying the 2D Euler equation
\begin{equation}%
\begin{cases}
\partial_{t}u+u\partial_{x}u+v\partial_{y}u=-\partial_{x}P\\
\partial_{t}v+u\partial_{x}v+v\partial_{y}v=-\partial_{y}P
\end{cases}
\label{Euler}%
\end{equation}
with the incompressibility condition
\begin{equation}
\partial_{x}u+\partial_{y}v=0 \label{incompre}%
\end{equation}
and the boundary conditions
\begin{equation}
v=0\text{ on }\left\{  y=-1\right\}  \text{ and }\left\{  y=1\right\}  .
\label{Euler-bc}%
\end{equation}
Here, $\vec{u}=\left(  u,v\right)  $ is the fluid velocity and $P$ is the
pressure. Define the vorticity $\omega=u_{y}-v_{x}$, then $\omega$ satisfies
the equation%
\[
\omega_{t}+u\omega_{x}+v\omega_{y}=0.
\]
Any shear flow $\left(  U\left(  y\right)  ,0\right)  \ $is a steady solution
for (\ref{Euler}). The Couette flow $\vec{u}_{0}=\left(  y,0\right)  $ is
among the simplest laminar flows, however, it poses several long-standing
puzzles in hydrodynamics. First, for any Reynolds number $R>0,\ $the Couette
flow is also a steady state for Navier-Stokes equations%
\begin{equation}%
\begin{cases}
\partial_{t}u+u\partial_{x}u+v\partial_{y}u=-\partial_{x}P+\frac{1}{R}\Delta
u\\
\partial_{t}v+u\partial_{x}v+v\partial_{y}v=-\partial_{y}P+\frac{1}{R}\Delta v
\end{cases}
\label{NS}%
\end{equation}
with (\ref{Euler-bc}) and the boundary conditions%
\[
\left(  u,v\right)  =\left(  \pm1,0\right)  \text{ on }\left\{  y=\pm
1\right\}  .
\]
The so called Sommerfeld paradox (\cite{Som08}) is that Couette flow is
linearly stable for any $R>0$ (proved in \cite{Rom73}), but it becomes
turbulent when $R$ is large as revealed in experiments and numerical
simulations. We refer to (\cite{li-lin}) and the references therein for
attempts to resolve this paradox. In this paper, we are interested in another
mystery about Couette flow, namely, the inviscid damping. It is obvious that
Couette flow is nonlinearly stable in any $L^{p}$ norm of vorticity $\omega$,
since for Couette flow $\omega_{0}=1$ and thus the vorticity perturbation is
preserved along the perturbed flow trajectory. In 1907, Orr (\cite{orr-1907})
observed that for the linearized Euler equation around Couette$,\ $the
vertical velocity $v\left(  t\right)  $ tends to zero when $t$ goes to
infinity. We refer to Section 4 for a more detailed study on the linear
damping of Couette flow. It is\ unusual that such damping phenomena can occur
for a time reversible system such as the Euler equation.
%ZL-426
Moreover, the issue of inviscid damping also appears in the study of many
other stable flows (\cite{caglioti}, \cite{levy-et-70}, \cite{pillai-Gould94},
\cite{schecter-et al}), and is believed to plan important roles on explaining
the appearance of coherent structures in 2D turbulence. To be precise
mathematically, the problem of nonlinear inviscid damping near Couette flow is
to prove or disprove the following statement: When the initial velocity is
close enough to Couette in the sense that%
\[
\left\Vert \left(  u\left(  0\right)  ,v\left(  0\right)  \right)  -\left(
y,0\right)  \right\Vert _{\mathbf{X}}\ \text{is small enough}%
\]
in some function space $\mathbf{X}$, then
\[
\left\Vert v\left(  t\right)  \right\Vert _{L^{2}}\rightarrow0\ \text{when\ }%
t\rightarrow\infty,
\]
that is, $\left(  u\left(  t\right)  ,v\left(  t\right)  \right)  $ tends
asymptotically to a shear flow $\left(  U_{\infty}\left(  y\right)  ,0\right)
$ near the Couette flow. So far, nonlinear inviscid damping has not been
proved for Couette flow or any other stable Euler flows. Our first result
shows that the minimal regularity for such nonlinear damping to be true is
$H^{\frac{5}{2}}$, that is, the velocity space $\mathbf{X}$ must be at least
$H^{\frac{5}{2}}$.

\begin{theorem}
\label{thm-existence}Fixed any $T>0$ and $0\leq s<\frac{3}{2},\ $then for any
$\varepsilon>0$, there exists a steady solution $\left(  u_{\varepsilon
}\left(  x,y\right)  ,v_{\varepsilon}\left(  x,y\right)  \right)  $ to Euler
equation (\ref{Euler}) with (\ref{incompre})-(\ref{Euler-bc}) such that
$\left(  u_{\varepsilon}\left(  x,y\right)  ,v_{\varepsilon}\left(
x,y\right)  \right)  $ has minimal $x-$period $T,$
\[
\left\Vert \omega_{\varepsilon}-1\right\Vert _{H_{\left(  0,T\right)
\times\left(  -1,1\right)  }^{s}}<\varepsilon,\text{ \ where\ \ }%
\omega_{\varepsilon}=\partial_{y}u_{\varepsilon}-\partial_{x}v_{\varepsilon},
\]
and $v_{\varepsilon}\left(  x,y\right)  $ is not identically zero.
\end{theorem}

The above Theorem immediately implies that nonlinear inviscid damping is
\textit{not} true in any (vorticity)$\ H^{s}\ \left(  s<\frac{3}{2}\right)
\ $neighborhood, or equivalently in any (velocity)$\ H^{s}\ \left(  s<\frac
{5}{2}\right)  \ $neighborhood of Couette flow. As a corollary of the proof of
Theorem \ref{thm-existence}, we also get the following structural instability
result for Couette flow.

\begin{corollary}
\label{cor-insta}Fixed any $T>0$ and $0\leq s<\frac{3}{2},\ $then for any
$\varepsilon>0$, there exists a shear flow $\left(  U_{\varepsilon}\left(
y\right)  ,0\right)  $ such that $\left\Vert U_{\varepsilon}^{\prime}\left(
y\right)  -1\right\Vert _{H^{s}\left(  -1,1\right)  }<\varepsilon$ and
$\left(  U_{\varepsilon}\left(  y\right)  ,0\right)  $ is exponentially
unstable to perturbations of $x-$period $T$.
\end{corollary}

The shear flow $\left(  U_{\varepsilon}\left(  y\right)  ,0\right)  $ is
unstable in the sense that unstable eigenvalues exist for the linearized
problem in the domain $\Omega_{T}=S_{T}\times\left(  -1,1\right)  $,\ where
$S_{T}$ is the $T-$periodic circle. By our results in
\cite{lin-zeng-euler-invar}, there exist stable and unstable manifolds near
$\left(  U_{\varepsilon}\left(  y\right)  ,0\right)  $ for the Euler equation
(\ref{Euler}) in $\Omega_{T}$. Therefore, Theorem \ref{thm-existence} and
Corollary \ref{cor-insta} imply that the long time dynamics in the (vorticity)
$H^{s}\ \left(  s<\frac{3}{2}\right)  $ neighborhood of Couette flow is very
rich, including nontrivial steady flows, stable and unstable manifolds of
nearby unstable shear flows.

Our next theorem shows that there exist \textit{no} nontrivial steadily
travelling flows in the (vorticity) $H^{s}\ \left(  s>\frac{3}{2}\right)  $
neighborhoods of Couette flow.

\begin{theorem}
\label{thm-non-existence}Fixed any $T>0,$ $s>\frac{3}{2},\ $there exists
$\varepsilon_{0}>0$ such that any travelling solution $\left(  u\left(
x-cy,y\right)  ,v\left(  x-cy,y\right)  \right)  $ $\left(  c\in
\mathbf{R}\right)  \ $to Euler equation (\ref{Euler})-(\ref{Euler-bc}) with
$x-$period $T$ and satisfying that
\[
\left\Vert \omega-1\right\Vert _{H_{\left(  0,T\right)  \times\left(
-1,1\right)  }^{s}}<\varepsilon_{0},
\]
must have $v\left(  x,y\right)  \equiv0$, that is, $\left(  u,v\right)  $ is
necessarily a shear flow.
\end{theorem}

By the proof of Theorem \ref{thm-non-existence}, we also have the following

\begin{corollary}
\label{cor-stable}Fixed any $T>0$ and $s>\frac{3}{2},\ $there exists
$\varepsilon_{0}>0$ such that any shear flow $\left(  U\left(  y\right)
,0\right)  $ satisfying
\[
\left\Vert U^{\prime}\left(  y\right)  -1\right\Vert _{H^{s}\left(
-1,1\right)  }\leq\varepsilon_{0},
\]
is linearly stable to perturbations of $x-$period $T$.
\end{corollary}

Theorem \ref{thm-non-existence} and Corollary \ref{cor-stable} suggests that
in the (vorticity) $H^{s}\left(  s>\frac{3}{2}\right)  \ $neighborhoods of
Couette flow, the long time dynamical behavior of Euler flows might be much
simpler. Particularly, the only steady structures in any reference frame are
nearby stable shear flows. A necessary condition for nonlinear inviscid
damping in any space is that there exist no nontrivial invariant structures
(time-periodic, quasi-periodic solutions etc.) near Couette flow in this
space. Theorem \ref{thm-linear} is a first step in this direction.

In Theorem \ref{thm-linear} in Section 4, we show that the linear decay holds
true for any initial vorticity in $L^{2}$ and the optimal decay rate is
already achieved for initial vorticity in $H^{1}$ (see Remark
\ref{rmk-decay rate}). This indicates that the contrasting dynamics in
$H^{s}\ $neighborhoods of Couette with $s<\frac{3}{2}$ or $s>\frac{3}{2}$ is a
truly nonlinear phenomena and it can not be traced back to the linear level.

A similar phenomena of collisionless damping for electron plasmas was
discovered at the linear level by Landau (\cite{landau}) in 1946. In the
physical literature, the collisionless damping had been often (\cite{ishenko},
\cite{pillai-Gould94}, \cite{levy-et-70}) compared with the inviscid damping
problem. In \cite{lin-zeng-VP-damping}, we obtained similar results for the
nonlinear Landau damping problem. Moreover, in the case of collisionless
plasmas we are able to prove a stronger result that $H^{\frac{3}{2}}$ is the
critical regularity for the existence of \textit{any} nontrivial invariant
structure near a stable homogeneous state.

This paper is organized as follows. In Section 2, we construct nontrivial
steady flows near Couette flow in (vorticity) $H^{s}\ \left(  s<\frac{3}%
{2}\right)  $ neighborhood and for any minimal $x-$period. In Section 3, the
non-existence of nontrivial travelling flows is proved in (vorticity)
$H^{s}\ \left(  s>\frac{3}{2}\right)  \ $neighborhood. Section 4 is to study
the linear damping problem in Sobolev spaces. Throughout this paper, we use
$C$ to denote a generic constant in the estimates and only indicate the
dependence of $C$ when it matters.

\section{Existence of Cat's-eyes in H$^{s}\ \left(  s<\frac{3}{2}\right)  $}

In this Section, we construct steady flows of Kelvin's cat's eyes structure
near Couette flow in the (vorticity) $H^{s}\ \left(  s<\frac{3}{2}\right)  $
space. Our strategy is to construct cat's eyes flows by bifurcation at
modified shear flows near Couette. We split the proof into several steps.

%CCZ

\begin{lemma}
\label{lemma-bifurcation}Assume $U\left(  y\right)  \in C^{5}\left[
-1,1\right]  ,\ $is odd, monotone in $\left[  -1,1\right]  $, and $U^{\prime
}(0)>0$. Let $Q\left(  y\right)  =\frac{U^{\prime\prime}\left(  y\right)
}{U\left(  y\right)  }$ and define the operator
\[
\mathcal{L}:=-\frac{d^{2}}{dy^{2}}+Q\left(  y\right)  ,\ H^{2}\left(
-1,1\right)  \rightarrow L^{2}\left(  -1,1\right)  ,
\]
with zero Dirichlet conditions at $\left\{  y=\pm1\right\}  $. If
$\mathcal{L}$ has a negative eigenvalue $-k_{0}^{2}$, then $\exists$
$\varepsilon_{0}>0$, such that for each $0<\varepsilon<\varepsilon_{0}$, there
exist a steady solution $\left(  u_{\varepsilon}\left(  x,y\right)
,v_{\varepsilon}\left(  x,y\right)  \right)  $ to Euler equations
(\ref{Euler})-(\ref{Euler-bc}) which has minimal period $T_{\varepsilon}$ in
$x$,
\[
\left\Vert \omega_{\varepsilon}\left(  x,y\right)  -U^{\prime}\left(
y\right)  \right\Vert _{H^{2}\left(  0,T_{\varepsilon}\right)  \times\left(
-1,1\right)  } = \varepsilon,
\]
and the streamlines of this steady flow near $y=0$ have cat's eyes structure,
with a leading order expression given by (\ref{cats-eye}). When $\varepsilon
\rightarrow0$, $T_{\varepsilon}\rightarrow\frac{2\pi}{k_{0}}$.
\end{lemma}

\begin{proof}
The proof is a slight modification of that in \cite{li-lin}. Let $\psi
_{0}\left(  y\right)  $ to be a stream function associated with the shear
$\left(  U\left(  y\right)  ,0\right)  $, i..e., $\psi_{0}^{\prime}\left(
y\right)  =U\left(  y\right)  $. Since $\psi_{0}\left(  y\right)  ,Q\left(
y\right)  $ are even in $\left[  -1,1\right]  $, we let $\psi_{0}\left(
y\right)  =G\left(  \frac{1}{2}y^{2}\right)  $ and $Q\left(  y\right)
=H\left(  \frac{1}{2}y^{2}\right)  $. Then
\[
G^{\prime}\left(  \frac{1}{2}y^{2}\right)  =\frac{\psi_{0}^{\prime}\left(
y\right)  }{y}=\frac{U\left(  y\right)  }{y}>0\text{, when }y\in\left[
-1,1\right]  ,
\]
and $G,\ H\in C^{1}$ because $U\left(  y\right)  \in C^{5}$. So we can define
a
%CCZ
function $f_{0}\in C^{2}\left[  \min\psi_{0},\max\psi_{0}\right]  $ such that
\[
f_{0}^{\prime}=H\circ G^{-1}\text{ and }f_{0}\left(  \psi_{0}\left(  0\right)
\right)  =\psi_{0}^{\prime\prime}\left(  0\right)  .
\]
Then we extend $f_{0}$ to $f\in C_{0}^{2}\left(  \mathbf{R}\right)  $ such
that $f=f_{0}$ in $\left[  \min\psi_{0},\max\psi_{0}\right]  $. By our
construction,
\begin{equation}
f^{\prime}(\psi_{0}\left(  y\right)  )=Q(y),\ \ \ \text{for }y\in\left[
-1,1\right]  \text{,} \label{eqn-f_0}%
\end{equation}
which implies that
\[
f^{\prime}(\psi_{0}\left(  y\right)  )\psi_{0}^{\prime}\left(  y\right)
=U^{\prime\prime}\left(  y\right)  =\psi_{0}^{\prime\prime\prime}\left(
y\right)  ,
\]
and an integration of above yields
\begin{equation}
f\left(  \psi_{0}\left(  y\right)  \right)  =\psi_{0}^{\prime\prime}\left(
y\right)  ,\ \text{for }y\in\left[  -1,1\right]  . \label{eqn-f-psi-0}%
\end{equation}
We construct steady flows near $\left(  U\left(  y\right)  ,0\right)  $ by
solving the elliptic equation
\[
\Delta\psi=f\left(  \psi\right)  ,
\]
where $\psi\left(  x,y\right)  $ is the stream function and $\left(
u,v\right)  =\left(  \psi_{y},-\psi_{x}\right)  $ is the steady velocity.
\ Let $\xi=\alpha x,\ \psi\left(  x,y\right)  =\tilde{\psi}\left(
\xi,y\right)  ,$ where $\tilde{\psi}\left(  \xi,y\right)  $ is $2\pi-$periodic
in $\xi.$ We use $\alpha^{2}$ as the bifurcation parameter. The equation for
$\tilde{\psi}\left(  \xi,y\right)  $ becomes
\begin{equation}
\alpha^{2}\frac{\partial^{2}\tilde{\psi}}{\partial\xi^{2}}+\frac{\partial
^{2}\tilde{\psi}}{\partial y^{2}}-f(\tilde{\psi})=0, \label{eqn-psi-tilde}%
\end{equation}
with the boundary conditions that $\tilde{\psi}$ takes constant values on
$\left\{  y=\pm1\right\}  $. Define the perturbation of the stream function
\[
\phi\left(  \xi,y\right)  =\tilde{\psi}\left(  \xi,y\right)  -\psi_{0}\left(
y\right)  .
\]
Then by using (\ref{eqn-f-psi-0}), we reduce the equation (\ref{eqn-psi-tilde}%
) to
\begin{equation}
\alpha^{2}\frac{\partial^{2}\phi}{\partial\xi^{2}}+\frac{\partial^{2}\phi
}{\partial y^{2}}-\left(  f(\phi+\psi_{0}\left(  y\right)  )-f\left(  \psi
_{0}\left(  y\right)  \right)  \right)  =0. \label{eqn-phi-traveling}%
\end{equation}
Define the spaces%
\[
B=\left\{  \phi(\xi,y)\in H^{3}([0,2\pi]\times\lbrack-1,1]),\text{ }\phi
(\xi,-1)=\phi(\xi,1)=0,\ 2\pi-\text{periodic and even in }\xi\right\}
\]
and
\[
D=\left\{  \phi(\xi,y)\in H^{1}([0,2\pi]\times\lbrack-1,1]),\text{ }%
2\pi-\text{periodic and even in }\xi\right\}  .
\]
Consider the mapping
\[
F(\phi,\alpha^{2})\ :B\times\mathbb{R}^{+}\mapsto D
\]
defined by
\[
F(\phi,\alpha^{2})=\alpha^{2}\frac{\partial^{2}\phi}{\partial\xi^{2}}%
+\frac{\partial^{2}\phi}{\partial y^{2}}-\left(  f(\phi+\psi_{0}\left(
y\right)  )-f\left(  \psi_{0}\left(  y\right)  \right)  \right)  .
\]
We study the bifurcation near the trivial solution $\phi=0$ of the equation
$F(\phi,\alpha^{2})=0$ in $B$, whose solutions give steady flows with
$x-$period $\frac{2\pi}{\alpha}$.\ The linearized operator of $F$
around$\ \left(  0,k_{0}^{2}\right)  $ has the form
\begin{align*}
\mathcal{G}  &  :=F_{\psi}(0,k_{0}^{2})=k_{0}^{2}\frac{\partial^{2}}%
{\partial\xi^{2}}+\frac{\partial^{2}}{\partial y}-f^{\prime}(\psi_{0}\left(
y\right)  )\\
&  =k_{0}^{2}\frac{\partial^{2}}{\partial\xi^{2}}+\frac{\partial^{2}}{\partial
y}-Q(y).
\end{align*}
%CCZ
By Strum-Liouville theory, all eigenvalues of ${\mathcal{L}}$ are simple. In
fact, as proven in Appendix of \cite{li-lin}, $-k_{0}^{2}$ $\ $is the only
negative eigenvalue of ${\mathcal{L}}$. Let $\phi_{0}(y)$ be the corresponding
positive eigenfunction. So the kernel of $\mathcal{G}:$ $\ B\mapsto D\ $is
given by
\[
\ker(\mathcal{G})=\left\{  \phi_{0}(y)\cos\xi\right\}  ,
\]
In particular, the dimension of $\ker$({$\mathcal{G}$}) is $1$. Since
{$\mathcal{G}$} is self-adjoint, $\phi_{0}(y)\cos\xi\not \in R(${$\mathcal{G}%
$}$)$ -- the range of {$\mathcal{G}$}. Notice that $\partial_{\alpha^{2}%
}\partial_{\phi}F(\phi,\alpha^{2})$ is continuous and
\[
\partial_{\alpha^{2}}\partial_{\phi}F(0,k_{0}^{2})\left(  \phi_{0}(y)\cos
\xi\right)  =\frac{\partial^{2}}{\partial\xi^{2}}\left[  \phi_{0}(y)\cos
\xi\right]  =-\phi_{0}(y)\cos\xi\not \in R({\mathcal{G}}).
\]
Therefore by the Crandall-Rabinowitz local bifurcation theorem \cite{CR71},
there exists a local bifurcating curve $\left(  \phi(\beta),\alpha^{2}%
(\beta)\right)  $ of $F(\phi,\alpha^{2})=0$, which intersects the trivial
curve $\left(  0,\alpha^{2}\right)  $ at $\alpha^{2}=k_{0}^{2}$, such that
\[
\phi(\beta)=\beta\phi_{0}(y)\cos\xi+o(\beta),
\]
$\alpha^{2}(\beta)$ is a continuous function of $\beta$, and $\alpha
^{2}(0)=k_{0}^{2}$. So the stream functions of the perturbed steady flows in
$(\xi,y)\ $coordinates take the form
\begin{equation}
\psi(\xi,y)=\psi_{0}\left(  y\right)  +\beta\phi_{0}(y)\cos\xi+o(\beta).
\label{cats-eye}%
\end{equation}
Since $\phi_{0}(y)>0$, $\psi_{0}^{\prime}\left(  0\right)  =U\left(  0\right)
=0$, the streamlines of perturbed flows have cat's eyes structure near
$\left\{  y=0\right\}  ,\ $with saddle points near $\left(  2\pi j,0\right)
$. The proof is completed.
\end{proof}

In the next lemma, we study the eigenvalue problem of $\mathcal{L\ }$for a
class of monotone shear flows near Couette flow. Let
\[
\operatorname{erf}\left(  x\right)  =\frac{2}{\sqrt{\pi}}\int_{0}^{x}%
e^{-s^{2}}ds,\ -\infty<x<+\infty,
\]
be the error function. For $\gamma>0,\ a>0,$ we define the shear profile
\begin{equation}
U_{\gamma,a}\left(  y\right)  =y+a\gamma^{2}\operatorname{erf}\left(  \frac
{y}{\gamma}\right)  ,\ \ y\in\left(  -1,1\right)  . \label{defn-shear}%
\end{equation}
Denote%
\begin{equation}
Q_{\gamma,a}\left(  y\right)  =\frac{U_{\gamma,a}^{\prime\prime}\left(
y\right)  }{U_{\gamma,a}\left(  y\right)  }, \label{defn-Q_gamma_a}%
\end{equation}
and
\[
\mathcal{L}_{\gamma,a}:H^{2}\left(  -1,1\right)  \rightarrow L^{2}\left(
-1,1\right)
\]
to be the operator $-\frac{d^{2}}{dy^{2}}+Q_{\gamma,a}\left(  y\right)  $ with
the Dirichlet conditions at $\left\{  y=\pm1\right\}  .$

\begin{lemma}
\label{lemma-eigenvalue}For any fixed $a>\frac{1}{2},\ $when $\gamma$ is small
enough, the operator $\mathcal{L}_{\gamma,a}$ has a unique negative eigenvalue
$-\beta_{\gamma,a}^{2}$. When $\gamma\rightarrow0,\ \beta_{\gamma,a}$ tends to
the unique root $\beta_{a}\ $of the equation
\begin{equation}
2a=\beta_{a}\coth\beta_{a}, \label{eigenvalue-formula}%
\end{equation}
with the error estimate $\left\vert \beta_{\gamma,a}-\beta_{a}\right\vert
=O\left(  \sqrt{\gamma}\right)  $.
\end{lemma}

\begin{proof}
We write the potential function $Q_{\gamma,a}\left(  y\right)  $ as
\begin{align*}
Q_{\gamma,a}\left(  y\right)   &  =-\frac{4a}{\gamma\sqrt{\pi}}\frac
{ye^{-\left(  \frac{y}{\gamma}\right)  ^{2}}}{y+a\gamma^{2}\operatorname{erf}%
\left(  \frac{y}{\gamma}\right)  }\\
&  =-\frac{4a}{\gamma\sqrt{\pi}}e^{-\left(  \frac{y}{\gamma}\right)  ^{2}%
}\frac{1}{1+\gamma a\operatorname{erf}\left(  \frac{y}{\gamma}\right)
/\left(  \left(  \frac{y}{\gamma}\right)  \right)  }\\
&  =-4a\frac{1}{\gamma}\sigma\left(  \frac{y}{\gamma}\right)  \frac
{1}{1+\gamma a\Lambda\left(  \frac{y}{\gamma}\right)  },
\end{align*}
where
\[
\sigma\left(  y\right)  =\frac{1}{\sqrt{\pi}}e^{-y^{2}},\ \Lambda\left(
y\right)  =\frac{\operatorname{erf}\left(  y\right)  }{y}.
\]
Since $\Lambda\left(  y\right)  \ $is positive and bounded, we formally derive
that
\[
Q_{\gamma,a}\left(  y\right)  \rightarrow-4a\delta\left(  0\right)
,\ \text{when }\gamma\rightarrow0.
\]
Thus, when $\gamma\rightarrow0,\ $the operator $\mathcal{L}_{\gamma,a}$ tends
to $-\frac{d^{2}}{dy^{2}}-4a\delta\left(  0\right)  $, for which the
eigenvalue can be calculated by the formula (\ref{eigenvalue-formula}). We
implement these ideas rigorously below. We divide the proof into several steps.

Step 1: Denote $\lambda_{\gamma,a}$ to be the lowest eigenvalue of
$\mathcal{L}_{\gamma,a}$ and $\phi_{\gamma,a}$ the corresponding eigenfunction
with $\left\Vert \phi_{\gamma,a}\right\Vert _{L^{2}}=1$. We show that for
$\gamma>0$ small enough,
\[
-16a^{2}\leq\lambda_{\gamma,a}<0,
\]
and
\begin{equation}
\left\Vert \phi_{\gamma,a}\right\Vert _{H^{1}}\leq8a+1.
\label{estimate-eigenfunction}%
\end{equation}
Note that
\[
\lambda_{\gamma,a}=\min_{\substack{\left\Vert \phi\right\Vert _{L^{2}}%
=1\\\phi\in H_{0}^{1}}}\left(  \mathcal{L}_{\gamma,a}\phi,\phi\right)
=\min_{\substack{\left\Vert \phi\right\Vert _{L^{2}}=1\\\phi\in H_{0}^{1}%
}}\left(  \left\Vert \phi^{\prime}\right\Vert _{L^{2}\left(  -1,1\right)
}^{2}+\int_{-1}^{1}Q_{\gamma,a}\left(  y\right)  \phi\left(  y\right)
^{2}dy\right)  .
\]
Let $\phi_{1}\left(  y\right)  =\left(  1-\left\vert y\right\vert \right)  $,
then when $\gamma$ is small enough,
\[
\lambda_{\gamma,a}\leq\frac{\left(  \mathcal{L}_{\gamma,a}\phi_{1},\phi
_{1}\right)  }{\left\Vert \phi_{1}\right\Vert _{L^{2}}^{2}}\leq\frac{3}%
{4}\left(  2-4a\right)  <0,
\]
since$\left\Vert \phi_{1}\right\Vert _{L^{2}}^{2}=\frac{2}{3}$ and
\[
\lim_{\gamma\rightarrow0+}\left(  \mathcal{L}_{\gamma,a}\phi_{1},\phi
_{1}\right)  =\left\Vert \phi_{1}^{\prime}\right\Vert _{L^{2}\left(
-1,1\right)  }^{2}-4a\phi_{1}\left(  0\right)  ^{2}=2-4a<0\text{. }%
\]
To estimate the lower bound of $\lambda_{\gamma,a}$, we take any $\phi\in
H_{0}^{1}\left(  -1,1\right)  $ with $\left\Vert \phi\right\Vert _{L^{2}}=1$,
then
\begin{align}
\left(  \mathcal{L}_{\gamma,a}\phi,\phi\right)   &  \geq\left\Vert
\phi^{\prime}\right\Vert _{L^{2}\left(  -1,1\right)  }^{2}-\int_{-1}^{1}%
\frac{4a}{\gamma\sqrt{\pi}}e^{-\left(  \frac{y}{\gamma}\right)  ^{2}%
}dy\left\Vert \phi\right\Vert _{L^{\infty}}^{2}\label{estimate-quadratic}\\
&  \geq\left\Vert \phi^{\prime}\right\Vert _{L^{2}\left(  -1,1\right)  }%
^{2}-4a\left\Vert \phi\right\Vert _{L^{\infty}}^{2}\nonumber\\
&  \geq\left\Vert \phi^{\prime}\right\Vert _{L^{2}\left(  -1,1\right)  }%
^{2}-8a\left\Vert \phi\right\Vert _{L^{2}}\left\Vert \phi^{\prime}\right\Vert
_{L^{2}\left(  -1,1\right)  }\geq-16a^{2}.\nonumber
\end{align}
Taking the minimum of above estimate, we get $\lambda_{\gamma,a}\geq-16a^{2}$.
Moreover, again from estimate (\ref{estimate-quadratic}),
\[
0>\lambda_{\gamma,a}=\left(  \mathcal{L}_{\gamma,a}\phi_{\gamma,a}%
,\phi_{\gamma,a}\right)  \geq\left\Vert \phi_{\gamma,a}^{\prime}\right\Vert
_{L^{2}}^{2}-8a\left\Vert \phi_{\gamma,a}^{\prime}\right\Vert _{L^{2}}%
\]
which implies that $\left\Vert \phi_{\gamma,a}^{\prime}\right\Vert _{L^{2}%
}\leq8a$.

Step 2: Let $\lambda_{a}$ be defined by
\begin{equation}
\lambda_{a}=\min_{\substack{\left\Vert \phi\right\Vert _{L^{2}}=1\\\phi\in
H_{0}^{1}}}\left\Vert \phi^{\prime}\right\Vert _{L^{2}}^{2}-4a\phi\left(
0\right)  ^{2}. \label{min-problem-limit}%
\end{equation}
%CCZ
We show that $\lambda_{a}=-\beta_{a}^{2}$ where $\beta_{a}$ solves the
equation (\ref{eigenvalue-formula}).

First, we claim that the minimum of (\ref{min-problem-limit}) is obtained at
some function $\phi_{a}\in H_{0}^{1}$. To show this claim, we note that by the
same estimates as in Step 1,
\[
-16a^{2}\leq\lambda_{a}<\frac{3}{2}\left(  2-4a\right)  <0.
\]
Let $\left\{  \phi_{n}\right\}  _{n=1}^{\infty}\subset H_{0}^{1}$ be a
minimizing sequence of (\ref{min-problem-limit}) with $\left\Vert \phi
_{n}\right\Vert _{L^{2}}=1$ and
\[
\left\Vert \phi_{n}^{\prime}\right\Vert _{L^{2}}^{2}-4a\phi_{n}\left(
0\right)  ^{2}\rightarrow\lambda_{a},\text{ when }n\rightarrow\infty.
\]
Similar to the estimate (\ref{estimate-eigenfunction}), when $n$ is large, we
have $\left\Vert \phi_{n}\right\Vert _{H^{1}}\leq8a+1.$ Thus $\phi
_{n}\rightarrow\phi_{a}$ weakly $H^{1},$ and strongly in $L^{2}\cap L^{\infty
}.$ Therefore, $\left\Vert \phi_{a}\right\Vert _{L^{2}}=1$ and
\[
\left\Vert \phi_{a}^{\prime}\right\Vert _{L^{2}}^{2}-4a\phi_{a}\left(
0\right)  ^{2}\leq\lim_{n\rightarrow\infty}\left\Vert \phi_{n}^{\prime
}\right\Vert _{L^{2}}^{2}-4a\phi_{n}\left(  0\right)  ^{2}=\lambda_{a}.
\]
Thus $\phi_{a}$ is the minimizer of (\ref{min-problem-limit}).

%CCZ
By taking the variation of (\ref{min-problem-limit}), one immediately obtains
that $\phi_{a}\in H_{0}^{1}$ satisfies the equation
\begin{equation}
\phi_{a}^{\prime\prime}+4a\phi_{a}(0)+\lambda_{a}\phi_{a}=0 \label{CalL_a}%
\end{equation}
in the sense of distribution. In particular, $\phi_{a}$ is continuous on
$[-1,1]$, $\phi_{a}(\pm1)=0$, and satisfies
\[
\phi_{a}^{\prime\prime}+\lambda_{a}\phi_{a}=0\text{, on }[-1,1]\backslash
\{0\}.
\]
Therefore, we have
\[
\phi_{a}(y)=\mp c\sinh\left(  \sqrt{-\lambda_{a}}(y\mp1)\right)  ,\quad\pm
y\in(0,1],
\]
for some constant $c$. To satisfy \eqref{CalL_a}, it is easy to check that
one
%LZ-426
must have
\[
\phi_{a}^{\prime}\left(  0+\right)  -\phi_{a}^{\prime}\left(  0-\right)
=-4a\phi_{a}(0),
\]
from which it follows that $\lambda_{a}=-\beta_{a}^{2}$ and $\beta_{a}$ solves
the equation (\ref{eigenvalue-formula}). Note that since the function
\[
f\left(  \beta\right)  =\beta\coth\beta:[0,\infty)\rightarrow\lbrack\frac
{1}{2},\infty)
\]
is monotone increasing. So for each $a>\frac{1}{2}$, there exists a unique
$\beta_{a}=f^{-1}\left(  a\right)  $ such that (\ref{eigenvalue-formula}) is satisfied.

%CCZ
Step 3: We show that when $\gamma$ is small enough,%
\begin{equation}
\left\vert \lambda_{\gamma,a}-\lambda_{a}\right\vert \leq C\left(  a\right)
\sqrt{\gamma}\text{.} \label{inequality-eigenvalues}%
\end{equation}
Denote the quadratic forms
\[
H_{\gamma,a}\left(  \phi\right)  =\left\Vert \phi^{\prime}\right\Vert
_{L^{2}\left(  -1,1\right)  }^{2}+\int_{-1}^{1}Q_{\gamma,a}\left(  y\right)
\phi\left(  y\right)  ^{2}dy
\]
and
\[
H_{a}\left(  \phi\right)  =\left\Vert \phi^{\prime}\right\Vert _{L^{2}\left(
-1,1\right)  }^{2}-4a\phi\left(  0\right)  ^{2}%
\]
in $H_{0}^{1}\left(  -1,1\right)  $. Then
\begin{align*}
\lambda_{a}  &  \leq H_{a}\left(  \phi_{\gamma,a}\right) \\
&  =H_{\gamma,a}\left(  \phi_{\gamma,a}\right)  +4a\int_{-1}^{1}\frac
{1}{\gamma}\sigma\left(  \frac{y}{\gamma}\right)  \frac{1}{1+\gamma
a\Lambda\left(  \frac{y}{\gamma}\right)  }\phi_{\gamma,a}^{2}\left(  y\right)
dy-4a\phi_{\gamma,a}\left(  0\right)  ^{2}\\
&  =\lambda_{\gamma,a}+4a\int_{-\frac{1}{\gamma}}^{\frac{1}{\gamma}}%
\sigma\left(  y\right)  \left(  \phi_{\gamma,a}^{2}\left(  \gamma y\right)
-\phi_{\gamma,a}^{2}\left(  0\right)  \right)  dy-4a\int_{\left\vert
y\right\vert \geq\frac{1}{\gamma}}\sigma\left(  y\right)  dy\phi_{\gamma
,a}^{2}\left(  0\right) \\
&  \ \ \ \ \ \ +4a\int_{-1}^{1}\frac{1}{\gamma}\sigma\left(  \frac{y}{\gamma
}\right)  \frac{\gamma a\Lambda\left(  \frac{y}{\gamma}\right)  }{1+\gamma
a\Lambda\left(  \frac{y}{\gamma}\right)  }\phi_{\gamma,a}^{2}\left(  y\right)
dy\\
&  =\lambda_{\gamma,a}+T_{1}+T_{2}+T_{3}.
\end{align*}
Since%
\begin{align*}
\left\vert \phi_{\gamma,a}^{2}\left(  \gamma y\right)  -\phi_{\gamma,a}%
^{2}\left(  0\right)  \right\vert  &  \leq2\left\Vert \phi_{\gamma
,a}\right\Vert _{L^{\infty}}\left\vert \int_{0}^{\gamma y}\phi_{\gamma
,a}^{\prime}\left(  s\right)  ds\right\vert \\
&  \leq C\left\Vert \phi_{\gamma,a}\right\Vert _{H^{1}\left(  -1,1\right)
}^{2}\sqrt{\gamma}\left\vert y\right\vert ^{\frac{1}{2}}\leq C\left(
8a+1\right)  \sqrt{\gamma}\left\vert y\right\vert ^{\frac{1}{2}},
\end{align*}
so%
\[
\left\vert T_{1}\right\vert \leq C\left(  a\right)  \sqrt{\gamma}%
\int_{\mathbf{R}}\sigma\left(  y\right)  \left\vert y\right\vert ^{\frac{1}%
{2}}dy\leq C\left(  a\right)  \sqrt{\gamma}.
\]
When $\gamma$ is small enough, we have
\[
\left\vert T_{2}\right\vert \leq C\left(  a\right)  \left\Vert \phi_{\gamma
,a}\right\Vert _{L^{\infty}}^{2}\int_{\left\vert y\right\vert \geq\frac
{1}{\gamma}}\sigma\left(  y\right)  dy\leq C\left(  a\right)  \sqrt{\gamma},
\]
and
\[
\left\vert T_{3}\right\vert \leq C\left(  a\right)  \left\Vert \phi_{\gamma
,a}\right\Vert _{L^{\infty}}^{2}\gamma\int_{\mathbf{R}}\sigma\left(  y\right)
dy\leq C\left(  a\right)  \sqrt{\gamma}.
\]
Thus
%LZ-426%
\[
\lambda_{a}-\lambda_{\gamma,a}\leq C\left(  a\right)  \sqrt{\gamma}%
\]
and similarly
\[
\lambda_{\gamma,a}-\lambda_{a}\leq C\left(  a\right)  \sqrt{\gamma}.
\]
%CCZ
This finishes the proof of (\ref{inequality-eigenvalues}) and thus also the lemma.
\end{proof}

We are now ready to prove Theorem \ref{thm-existence}.

\begin{proof}
[Proof of Theorem \ref{thm-existence}]Fixed $T>0$, there exists $\frac{1}%
{2}<a_{1}<a_{2}$ such that
\[
\beta_{a_{1}}<\frac{2\pi}{T}<\beta_{a_{2}}\text{.}%
\]
By Lemma \ref{lemma-eigenvalue}, there exists $\gamma_{0}>0$ small enough,
such that when $0<\gamma<\gamma_{0}$, for all $a\in\left(  a_{1},a_{2}\right)
\ $the operator $\mathcal{L}_{\gamma,a}$ has a negative eigenvalue
$\lambda_{\gamma,a}$ and
\begin{equation}
\sqrt{-\lambda_{\gamma,a_{1}}}<\frac{2\pi}{T}<\sqrt{-\lambda_{\gamma,a_{2}}%
}\text{.} \label{inequality-eigen}%
\end{equation}
We show that: for $a\in\left(  a_{1},a_{2}\right)  ,$ $s\in\lbrack0,\frac
{3}{2}),$
\begin{equation}
\left\Vert U_{\gamma,a}^{\prime}\left(  y\right)  -1\right\Vert _{H^{s}\left(
-1,1\right)  }\rightarrow0,\text{ when }\gamma\rightarrow0. \label{limit-H-s}%
\end{equation}
Indeed,
\[
U_{\gamma,a}^{\prime}\left(  y\right)  -1=\frac{2a\gamma}{\sqrt{\pi}%
}e^{-\left(  \frac{y}{\gamma}\right)  ^{2}},
\]
so
\[
\left\Vert U_{\gamma,a}^{\prime}\left(  y\right)  -1\right\Vert _{H^{s}\left(
-1,1\right)  }\leq C\left\Vert \gamma e^{-\left(  \frac{y}{\gamma}\right)
^{2}}\right\Vert _{H^{s}\left(  \mathbf{R}\right)  }.
\]
%CCZ
Using the Fourier transform, one may compute explicitly
\[
\left\Vert \gamma e^{-\left(  \frac{y}{\gamma}\right)  ^{2}}\right\Vert
_{\dot{H}^{s}\left(  \mathbf{R}\right)  }=C_{s}\gamma^{\frac{3}{2}-s},\
\]
which implies (\ref{limit-H-s}) by our assumption that $s<\frac{3}{2}$.
Thus For any $\varepsilon>0$, by choosing $\gamma_{0}$ small enough, we can
assume that
\begin{equation}
\left\Vert U_{\gamma,a}^{\prime}\left(  y\right)  -1\right\Vert _{H^{s}\left(
-1,1\right)  }\leq\frac{\varepsilon}{2T},\text{ when }\left(  \gamma,a\right)
\in\left(  0,\gamma_{0}\right)  \times\left(  a_{1},a_{2}\right)  .
\label{inequlity-shear}%
\end{equation}
By Lemma \ref{lemma-bifurcation}, for any $\left(  \gamma,\delta\right)
\in\left(  0,\gamma_{0}\right)  \times\left(  a_{1},a_{2}\right)  \,,$ there
exists local bifurcation of non-parallel steady flows (Cats's eyes) of Euler
equation (\ref{Euler})-(\ref{Euler-bc}), near the shear flow $\left(
U_{\gamma,a}\left(  y\right)  ,0\right)  $. For each fixed $0<\gamma
<\gamma_{0},\ $we can find $r_{0}>0$ (independent of $a\in\left(  a_{1}%
,a_{2}\right)  \ $) such that for any $0<r<r_{0}\,$, there exists a nontrivial
steady solution
\[
\left(  u_{\gamma,a;r}\left(  x,y\right)  ,v_{\gamma,a;r}\left(  x,y\right)
\right)
\]
with vorticity $\omega_{\gamma,a;r}\left(  x,y\right)  $ which has $x-$period
$T\left(  \gamma,a;r\right)  $ and
\[
\left\Vert \omega_{\gamma,a;r}-U_{\gamma,a}^{\prime}\left(  y\right)
\right\Vert _{H^{2}\left(  0,T\left(  \gamma,a;r\right)  \right)
\times\left(  -1,1\right)  }=r.
\]
Moreover,
\[
\frac{2\pi}{T\left(  \gamma,a;r\right)  }\rightarrow\sqrt{-\lambda_{\gamma,a}%
}\text{, when }r\rightarrow0.
\]
By (\ref{inequality-eigen}), when $r_{0}$ is small enough,
\[
T\left(  \gamma,a_{1};r\right)  <T<T\left(  \gamma,a_{2};r\right)  ,\text{ for
}0<r<r_{0}.
\]
Since $T\left(  \gamma,a;r\right)  $ is continuous to $a,$ for each $\gamma
\in\left(  0,\gamma_{0}\right)  $ and $r>0$ small enough, there exists
$a_{T}\left(  \gamma,r\right)  \in\left(  a_{1},a_{2}\right)  \,$, such that
$T\left(  \gamma,a_{T};r\right)  =T$. Then the flow
\[
\left(  u_{\gamma;r}\left(  x,y\right)  ,v_{\gamma;r}\left(  x,y\right)
\right)  :=\left(  u_{\gamma,a_{T};r}\left(  x,y\right)  ,v_{\gamma,a_{T}%
;r}\left(  x,y\right)  \right)
\]
with the vorticity $\omega_{\gamma;r}=\omega_{\gamma,a_{T};r}$ is a nontrivial
steady solution of Euler equation, with $x-$period $T$ and
\[
\left\Vert \omega_{\gamma;r}-U_{\gamma,a_{T}}^{\prime}\left(  y\right)
\right\Vert _{H^{2}\left(  0,T\right)  \times\left(  -1,1\right)  }=r.
\]
Thus for any $0<r<\min\left\{  \gamma_{0},\frac{\varepsilon}{2}\right\}  $,
combining with (\ref{inequlity-shear}) we have
\[
\left\Vert \omega_{\gamma;r}\left(  x,y\right)  -1\right\Vert _{H^{s}\left(
0,T\right)  \times\left(  -1,1\right)  }<\varepsilon.
\]
This finishes the proof of Theorem \ref{thm-existence}.
\end{proof}

For the shear flow $U_{\gamma,a}\left(  y\right)  \ $defined by
(\ref{defn-shear}), there is only one inflection point at $y=0$. The following
Lemma about linear instability of $\left(  U_{\gamma,a}\left(  y\right)
,0\right)  \ $follows from the result in \cite{lin-siam}.

\begin{lemma}
\label{lemma-insta-shear}If the operator $\mathcal{L}_{\gamma,a}$ has a
negative eigenvalue $\lambda_{\gamma,a}<0$, then the shear flow $\left(
U_{\gamma,a}\left(  y\right)  ,0\right)  $ is linearly exponentially unstable
to perturbations of any $x-$period greater than $\frac{2\pi}{\sqrt
{-\lambda_{\gamma,a}}}$.
\end{lemma}

From above Lemma, it is easy to prove Corollary \ref{cor-insta}.

\begin{proof}
[Proof of Corollary \ref{cor-insta}]For any fixed $T>0$, pick $a>\frac{1}{2}$
such that $T>\frac{2\pi}{\beta_{a}}.$Then there exists $\gamma$ small enough
such that
\[
\lambda_{\gamma,a}<0,\ T>\frac{2\pi}{\sqrt{-\lambda_{\gamma,a}}}.
\]
By Lemma \ref{lemma-insta-shear}, the shear flow $\left(  U_{\gamma,a}\left(
y\right)  ,0\right)  $ is linearly exponentially unstable to perturbations of
$x-$period $T$. For any $\varepsilon>0,\ $if $\gamma$ is small enough, by
(\ref{limit-H-s}) we can let
\[
\left\Vert U_{\gamma,a}^{\prime}\left(  y\right)  -1\right\Vert _{H^{s}\left(
-1,1\right)  }<\varepsilon
\]
and this finishes the proof.
\end{proof}

\begin{remark}
We can use more general shear profiles than $U_{\gamma,a}\left(  y\right)  $
in (\ref{defn-shear}) to construct cats's eyes flows near Couette. More
precisely, define
\[
U_{\gamma,a}\left(  y\right)  =y+a\gamma^{2}h\left(  \frac{y}{\gamma}\right)
,
\]
where $h\in C^{5}\left(  \mathbf{R}\right)  $ is odd, $h^{\prime}\in
H^{2}\left(  \mathbf{R}\right)  ,$ and
\[
\int_{\mathbf{R}}\frac{h^{\prime\prime}\left(  x\right)  }{x}dx=b_{0}%
>0,\ \ \ a>\frac{2}{b_{0}}.
\]
By the same proof of Lemma \ref{lemma-eigenvalue}, when $\gamma$ is small
enough, the operator
\[
\mathcal{L}_{\gamma,a}:=-\frac{d^{2}}{dy^{2}}+Q_{\gamma,a}\left(  y\right)
,\ \text{with }Q_{\gamma,a}\left(  y\right)  =\frac{U_{\gamma,a}^{\prime
\prime}\left(  y\right)  }{U_{\gamma,a}\left(  y\right)  },
\]
has a negative eigenvalue $-\beta_{\gamma,a}^{2}$, where
\[
\left\vert \beta_{\gamma,a}-\beta_{a}\right\vert =O\left(  \sqrt{\gamma
}\right)  \text{ and }\frac{b_{0}a}{2}=\beta_{a}\coth\beta_{a}.
\]
Then the same proof of Theorem \ref{thm-existence} yields cats's eyes flows
bifurcating form $\left(  U_{\gamma,a}\left(  y\right)  ,0\right)  .$ Such
shear flows $\left(  U_{\gamma,a}\left(  y\right)  ,0\right)  $ are
exponentially unstable for perturbations with $x-$period $T$ near $\frac{2\pi
}{\beta_{\gamma,a}}$.
\end{remark}

\section{Non-existence of traveling waves in H$^{s}\left(  s>\frac{3}%
{2}\right)  $}

In this Section, we prove Theorem \ref{thm-non-existence}. For the proof, we
need a few lemmas. The first lemma is a Hardy type inequality.

\begin{lemma}
\label{lemma-Hardy}Let $s\in\left(  \frac{1}{2},\frac{3}{2}\right)  $. If
$u\left(  y\right)  \in H^{s}\left(  \mathbf{-}1,1\right)  ,\ $and $u\left(
y_{0}\right)  =0$ for some $y_{0}\in\left[  -1,1\right]  ,\ $then for any
$1\leq p<\frac{1}{\frac{3}{2}-s},$%
\[
\left\Vert \frac{u\left(  y\right)  }{y-y_{0}}\right\Vert _{L^{p}\left(
-1,1\right)  }\leq C\left(  p\right)  \left\Vert u\right\Vert _{H^{s}\left(
\mathbf{-}1,1\right)  }.
\]

\end{lemma}

\begin{proof}
Since$\ s>\frac{1}{2}$, the space $H^{s}\left(  -1,1\right)  $ is embedded to
the H\"{o}lder space $C^{0,s-\frac{1}{2}}\left(  -1,1\right)  $. So
\[
\left\vert u\left(  y\right)  \right\vert =\left\vert u\left(  y\right)
-u\left(  y_{0}\right)  \right\vert \leq\left\vert y-y_{0}\right\vert
^{s-\frac{1}{2}}\left\Vert u\right\Vert _{C^{0,\alpha}}\leq C\left\vert
y-y_{0}\right\vert ^{s-\frac{1}{2}}\left\Vert u\right\Vert _{H^{s}},
\]
and
\begin{align*}
\int_{-1}^{1}\left\vert \frac{u\left(  y\right)  }{y-y_{0}}\right\vert ^{p}dy
&  \leq\left\Vert u\right\Vert _{H^{s}}^{p}\int_{-1}^{1}\frac{1}{\left\vert
y-y_{0}\right\vert ^{\left(  \frac{3}{2}-s\right)  p}}dv\\
&  =\frac{1}{1-\left(  \frac{3}{2}-s\right)  p}\left(  \left(  1-y_{0}\right)
^{1-\left(  \frac{3}{2}-s\right)  p}+\left(  y_{0}+1\right)  ^{1-\left(
\frac{3}{2}-s\right)  p}\right)  \left\Vert u\right\Vert _{H^{s}}^{p}\\
&  \leq\frac{1}{1-\left(  \frac{3}{2}-s\right)  p}2^{1-\left(  \frac{3}%
{2}-s\right)  p}\left\Vert u\right\Vert _{H^{s}}^{p}.
\end{align*}
This finishes the proof.
\end{proof}

\begin{proof}
[Proof of Theorem \ref{thm-non-existence}]Suppose otherwise, then there exist
a sequence $\varepsilon_{n}\rightarrow0$, and travelling solutions $\left(
u_{n}\left(  x-c_{n}t,y\right)  ,v_{n}\left(  x-c_{n}t,y\right)  \right)  $ to
Euler equation (\ref{Euler})-(\ref{Euler-bc}) which are $T-$periodic in $x$
and such that $v_{n}$ is not identically zero,
\begin{equation}
\left\Vert \omega_{n}-1\right\Vert _{H_{\left(  0,T\right)  \times\left(
-1,1\right)  }^{s}}<\varepsilon_{n},\text{ where }\omega_{n}\left(
x,y\right)  =\partial_{y}u_{n}-\partial_{x}v_{n}\text{. }
\label{estimate-vor-steady}%
\end{equation}
We can assume that
\begin{equation}
\int_{0}^{T}\int_{-1}^{1}u_{n}\left(  x,y\right)  dydx=0,
\label{condition-zero-u}%
\end{equation}
otherwise we consider the travelling wave
\[
u_{n}\left(  x-\left(  c_{n}+d_{n}\right)  t,y\right)  -d_{n},v_{n}\left(
x-\left(  c_{n}+d_{n}\right)  t,y\right)  ,
\]
with
\[
d_{n}=\frac{1}{2T}\int_{0}^{T}\int_{-1}^{1}u_{n}\left(  x,y\right)  dydx.
\]
The travelling wave solutions satisfy the vorticity equation%
\begin{equation}
\left(  u_{n}-c_{n}\right)  \partial_{x}\omega_{n}+v_{n}\partial_{y}\omega
_{n}=0. \label{eqn-vor-steady}%
\end{equation}
Because of the condition (\ref{condition-zero-u}), $\left(  u_{n}%
,v_{n}\right)  $ is uniquely determined by the vorticity $\omega_{n}$ and
\[
\left\Vert \left(  u_{n},v_{n}\right)  -\left(  y,0\right)  \right\Vert
_{H_{\left(  0,T\right)  \times\left(  -1,1\right)  }^{s+1}}\leq C\left\Vert
\omega_{n}-1\right\Vert _{H_{\left(  0,T\right)  \times\left(  -1,1\right)
}^{s}}\leq C\varepsilon_{n}.
\]
Since $s>\frac{3}{2},$
\[
\left\Vert \partial_{y}u_{n}-1\right\Vert _{L^{\infty}\left(  0,T\right)
\times\left(  -1,1\right)  }\leq\left\Vert u_{n}-y\right\Vert _{H_{\left(
0,T\right)  \times\left(  -1,1\right)  }^{s+1}}\leq C\varepsilon_{n},
\]
thus when $n$ is large,
\begin{equation}
\frac{1}{2}<\partial_{y}u_{n}<\frac{3}{2}\text{, in }\left(  0,T\right)
\times\left[  -1,1\right]  . \label{estimate-y-derivative}%
\end{equation}
Therefore, for each $x\in\left(  0,T\right)  $, $u_{n}\left(  x,y\right)  $ is
strictly increasing for $y\in\left[  -1,1\right]  $. We divide $\left(
0,T\right)  $ into three subsets%
\[
P_{n}=\left\{  x\ |\ c_{n}\leq u_{n}\left(  x,-1\right)  \right\}
,\ Q_{n}=\left\{  x\ |\ c_{n}\geq u_{n}\left(  x,1\right)  \right\}  ,
\]
and
\[
S_{n}=\left\{  x\ |\ u_{n}\left(  x,-1\right)  <c_{n}<u_{n}\left(  x,1\right)
\right\}  .
\]
When $x\in S_{n}$, there exists a unique $y_{n}\left(  x\right)  \in\left(
-1,1\right)  $ such that $u_{n}\left(  x,y_{n}\left(  x\right)  \right)
=c_{n}$. From (\ref{eqn-vor-steady}), it follows that
\[
v_{n}\left(  x,y_{n}\left(  x\right)  \right)  =0\text{ or }\partial_{y}%
\omega_{n}\left(  x,y_{n}\left(  x\right)  \right)  =0,
\]
and we further divide $S_{n}$ into two subsets
\begin{align*}
S_{n}^{1}  &  =\left\{  x\in S_{n}\ |\ v_{n}\left(  x,y_{n}\left(  x\right)
\right)  =0\text{ }\right\}  ,\ \\
S_{n}^{2}  &  =\left\{  x\in S_{n}\ |\ \partial_{y}\omega_{n}\left(
x,y_{n}\left(  x\right)  \right)  =0\text{ }\right\}  .
\end{align*}
By the incompressible condition (\ref{incompre}),
\[
\partial_{x}\omega_{n}=\partial_{x}\left(  \partial_{y}u_{n}-\partial_{x}%
v_{n}\right)  =-\Delta v_{n}.
\]
Since $v_{n}\left(  x,\pm1\right)  =0$ by (\ref{Euler-bc}), by integration by
parts and using (\ref{eqn-vor-steady}), we get%
\begin{align}
\int_{0}^{T}\int_{-1}^{1}\left\vert \nabla v_{n}\right\vert ^{2}\ dydx  &
=\int_{0}^{T}\int_{-1}^{1}v_{n}\partial_{x}\omega_{n}\ dydx=-\int_{0}^{T}%
\int_{-1}^{1}v_{n}\frac{v_{n}\partial_{y}\omega_{n}}{u_{n}-c_{n}%
}\ dydx\label{inequlity-steady-variational}\\
&  \leq\int_{P_{n}}\int_{-1}^{1}\left\vert v_{n}\frac{v_{n}}{u_{n}-c_{n}%
}\ \partial_{y}\omega_{n}\right\vert dydx+\int_{Q_{n}}\int_{-1}^{1}\left\vert
v_{n}\frac{v_{n}}{u_{n}-c_{n}}\ \partial_{y}\omega_{n}\right\vert
dydx\nonumber\\
&  \ \ \ +\int_{S_{n}^{1}}\int_{-1}^{1}\left\vert v_{n}\frac{v_{n}}%
{u_{n}-c_{n}}\ \partial_{y}\omega_{n}\right\vert dydx+\int_{S_{n}^{2}}%
\int_{-1}^{1}\left\vert v_{n}^{2}\frac{\partial_{y}\omega_{n}}{u_{n}-c_{n}%
}\ \right\vert dydx\nonumber\\
&  =I+II+III+IV.\nonumber
\end{align}
Since $\left(  \frac{\pi}{2}\right)  ^{2}$ is the lowest eigenvalue of
$-\Delta$ on $\left(  0,T\right)  \times\left(  -1,1\right)  $ with periodic
boundary condition in $x$ and Dirichlet boundary condition in $y$,
\[
\left\Vert \nabla v_{n}\right\Vert _{L_{\left(  0,T\right)  \times\left(
-1,1\right)  }^{2}}^{2}\geq\left(  \frac{\pi}{2}\right)  ^{2}\left\Vert
v_{n}\right\Vert _{L_{\left(  0,T\right)  \times\left(  -1,1\right)  }^{2}%
}^{2}.
\]
Thus by Sobolev embedding, for any $p>1,$
\begin{equation}
\left\Vert v_{n}\right\Vert _{L_{\left(  0,T\right)  \times\left(
-1,1\right)  }^{p}}\leq C\left(  p\right)  \left\Vert v_{n}\right\Vert
_{H_{\left(  0,T\right)  \times\left(  -1,1\right)  }^{1}}\leq C\left(
p\right)  \left\Vert \nabla v_{n}\right\Vert _{L_{\left(  0,T\right)
\times\left(  -1,1\right)  }^{2}}. \label{estimate-v-n}%
\end{equation}
Since
\begin{equation}
\left\Vert \partial_{y}\omega_{n}\right\Vert _{H_{\left(  0,T\right)
\times\left(  -1,1\right)  }^{s-1}}\leq\left\Vert \omega_{n}-1\right\Vert
_{H_{\left(  0,T\right)  \times\left(  -1,1\right)  }^{s}}<\varepsilon_{n},
\label{estimate-par-vort}%
\end{equation}
again by Sobolev embedding,
\[
\left\Vert \partial_{y}\omega_{n}\right\Vert _{L_{\left(  0,T\right)
\times\left(  -1,1\right)  }^{p}}\leq C\left(  p\right)  \varepsilon_{n}\text{
for any }1<p<\frac{2}{\left(  2-s\right)  _{+}},
\]
where
\[
\left(  2-s\right)  _{+}=\max\left\{  2-s,0\right\}  .
\]
So we can always choose $p_{1},p_{2},p_{3}$ such that
\[
p_{1}>1,\ \ 1<p_{2}<\frac{2}{\left(  2-s\right)  _{+}},\ \ 1<p_{3}<2,
\]
\ and
\[
\frac{1}{p_{1}}+\frac{1}{p_{2}}+\frac{1}{p_{3}}=1.
\]
When $x\in P_{n},$\
\[
\left\vert u_{n}\left(  x,y\right)  -c_{n}\right\vert \geq\left\vert
u_{n}\left(  x,y\right)  -u_{n}\left(  x,-1\right)  \right\vert ,
\]
so$\ $
\begin{align*}
I  &  \leq\left\Vert v_{n}\right\Vert _{L_{\left(  0,T\right)  \times\left(
-1,1\right)  }^{p_{1}}}\left\Vert \partial_{y}\omega_{n}\right\Vert
_{L_{\left(  0,T\right)  \times\left(  -1,1\right)  }^{p_{2}}}\left\Vert
\frac{v_{n}}{y+1}\right\Vert _{L_{P_{n}\times\left(  -1,1\right)  }^{p_{3}}%
}\left\Vert \frac{y+1}{u_{n}\left(  x,y\right)  -u_{n}\left(  x,-1\right)
}\right\Vert _{L_{\left(  0,T\right)  \times\left(  -1,1\right)  }^{\infty}}\\
&  \leq C\varepsilon_{n}\left\Vert \nabla v_{n}\right\Vert _{L_{\left(
0,T\right)  \times\left(  -1,1\right)  }^{2}}\left(  \int_{P_{n}}\left\Vert
\frac{v_{n}}{y+1}\right\Vert _{L^{p_{3}}\left(  -1,1\right)  }^{p_{3}%
}dx\right)  ^{\frac{1}{p_{3}}}\text{ }\\
&  \leq C\varepsilon_{n}\left\Vert \nabla v_{n}\right\Vert _{L_{\left(
0,T\right)  \times\left(  -1,1\right)  }^{2}}\left(  \int_{P_{n}}\left\Vert
v_{n}\right\Vert _{H^{1}\left(  -1,1\right)  }^{p_{3}}\ dx\right)  ^{\frac
{1}{p_{3}}}\text{ (By Lemma \ref{lemma-Hardy})}\\
&  \leq C\varepsilon_{n}\left\Vert \nabla v_{n}\right\Vert _{L_{\left(
0,T\right)  \times\left(  -1,1\right)  }^{2}}\left\Vert v_{n}\right\Vert
_{H_{\left(  0,T\right)  \times\left(  -1,1\right)  }^{1}}T^{\frac{1}{p_{3}%
}-\frac{1}{2}}\leq C\varepsilon_{n}\left\Vert \nabla v_{n}\right\Vert
_{L_{\left(  0,T\right)  \times\left(  -1,1\right)  }^{2}}^{2}.
\end{align*}
Here in the second inequality above, we use (\ref{estimate-v-n}),
(\ref{estimate-par-vort}) and the estimate
\[
\left\vert \frac{y+1}{u_{n}\left(  x,y\right)  -u_{n}\left(  x,-1\right)
}\right\vert =\frac{1}{\left\vert \partial_{y}u_{n}\left(  x,\tilde{y}\right)
\right\vert }\leq2,\ \text{ }\tilde{y}\in\left(  -1,y\right)  \text{,\ }%
\]
due to (\ref{estimate-y-derivative}). By similar estimates as that for $I$, we
get
\[
II,\ III\leq C\varepsilon_{n}\left\Vert \nabla v_{n}\right\Vert _{L_{\left(
0,T\right)  \times\left(  -1,1\right)  }^{2}}^{2}.
\]
To estimate $IV,$ we choose
%CCZ%
\[
1<p_{1}<\min\{2,\frac{1}{\left(  \frac{5}{2}-s\right)  _{+}}\},\ \ p_{2}%
=2p_{1}^{\prime},
\]
then
\begin{align*}
IV  &  \leq C\left\Vert v_{n}\right\Vert _{L_{\left(  0,T\right)
\times\left(  -1,1\right)  }^{p2}}^{2}\left\Vert \frac{\partial_{y}\omega_{n}%
}{y-y_{n}\left(  x\right)  }\ \right\Vert _{L_{S_{n}^{2}\times\left(
-1,1\right)  }^{p_{1}}}\\
&  \leq C\left\Vert \nabla v_{n}\right\Vert _{L_{\left(  0,T\right)
\times\left(  -1,1\right)  }^{2}}^{2}\left(  \int_{S_{n}^{2}}\left\Vert
\partial_{y}\omega_{n}\right\Vert _{H^{s-1}\left(  -1,1\right)  }^{p_{1}%
}\ dx\right)  ^{\frac{1}{p_{1}}}\text{ (By Lemma \ref{lemma-Hardy})}\\
&  \leq C\left\Vert \nabla v_{n}\right\Vert _{L_{\left(  0,T\right)
\times\left(  -1,1\right)  }^{2}}^{2}\left\Vert \omega_{n}-1\right\Vert
_{H_{\left(  0,T\right)  \times\left(  -1,1\right)  }^{s}}\leq C\varepsilon
_{n}\left\Vert \nabla v_{n}\right\Vert _{L_{\left(  0,T\right)  \times\left(
-1,1\right)  }^{2}}^{2}.
\end{align*}
Thus from (\ref{inequlity-steady-variational}) and above estimates,
\[
\left\Vert \nabla v_{n}\right\Vert _{L_{\left(  0,T\right)  \times\left(
-1,1\right)  }^{2}}^{2}\leq C\varepsilon_{n}\left\Vert \nabla v_{n}\right\Vert
_{L_{\left(  0,T\right)  \times\left(  -1,1\right)  }^{2}}^{2},
\]
When $n$ is large,\ this implies that $\nabla v_{n}=0$ and thus $v_{n}=0$.
This is a contradiction.
\end{proof}

To prove Corollary \ref{cor-stable}, we use the following Lemma which follows
from Theorem 2.7 of \cite{lin-comt}.

\begin{lemma}
\label{lemma-monotone}Let $U\left(  y\right)  \in C^{2}\left[  -1,1\right]  $
be a monotone flow. Denote $U_{s}^{1},\cdots,U_{s}^{l}$ to be all the
inflection values of $U\left(  y\right)  ,$ that is, $U_{s}^{i}=U\left(
y^{i}\right)  $ for some $y^{i}\in\left[  -1,1\right]  $ satisfying
$U^{\prime\prime}\left(  y^{i}\right)  =0$. Then the shear flow $\left(
U\left(  y\right)  ,0\right)  $ is linearly stable to perturbations of
$x-$period $T$, if for any $1\leq i\leq l$, the operator
\[
L_{i}=-\frac{d^{2}}{dy^{2}}+\frac{U^{\prime\prime}}{U-U_{s}^{i}}%
\]
with Dirichlet boundary conditions in $\left[  -1,1\right]  $ has the lowest
eigenvalue greater than $-\left(  \frac{2\pi}{T}\right)  ^{2}$.
\end{lemma}

\begin{proof}
[Proof of Corollary \ref{cor-stable}]We use the notations in Lemma
\ref{lemma-monotone}. We shall show that $L_{i}>0$ for any $1\leq i\leq l,$
when $\left\Vert U^{\prime}\left(  y\right)  -1\right\Vert _{H^{s}\left(
-1,1\right)  }\leq\varepsilon_{0}$ $\left(  s>\frac{3}{2}\right)  \ $is
sufficiently small. Then the conclusion of Corollary \ref{cor-stable} follows
from Lemma \ref{lemma-monotone}. Take any nonzero function $u\in H_{0}%
^{1}\left(  -1,1\right)  $, then
\[
\left(  L_{i}u,u\right)  =\left\Vert u^{\prime}\right\Vert _{L^{2}\left(
-1,1\right)  }^{2}+\int_{-1}^{1}\frac{U^{\prime\prime}}{U-U_{s}^{i}}u\left(
y\right)  ^{2}dy.
\]
Fix
\[
1<p_{1}<\min\{2,\frac{1}{\left(  \frac{5}{2}-s\right)  _{+}}\}
\]
and let $p_{2}=2p_{1}^{\prime}$. Since $u\left(  \pm1\right)  =0,\ $
\[
\left\Vert u^{\prime}\right\Vert _{L^{2}\left(  -1,1\right)  }\geq\frac{\pi
}{2}\left\Vert u\right\Vert _{L^{2}\left(  -1,1\right)  },
\]
and by Sobolev embedding%
\[
\left\Vert u\right\Vert _{L^{p_{2}}\left(  -1,1\right)  }\leq C\left\Vert
u^{\prime}\right\Vert _{L^{2}\left(  -1,1\right)  }.
\]
When $\varepsilon_{0}$ is small enough,
\[
\frac{1}{2}<U^{\prime}\left(  y\right)  <\frac{3}{2},\text{ for }y\in\left[
-1,1\right]  ,
\]
thus by Lemma \ref{lemma-Hardy},
\begin{align*}
\left(  L_{i}u,u\right)   &  \geq\left\Vert u^{\prime}\right\Vert _{L^{2}}%
^{2}-\int_{-1}^{1}\left\vert \frac{U^{\prime\prime}}{y-y^{i}}\right\vert
u\left(  y\right)  ^{2}dy\ \left\Vert \frac{y-y^{i}}{U-U_{s}^{i}}\right\Vert
_{L^{\infty}}\\
&  \geq\left\Vert u^{\prime}\right\Vert _{L^{2}}^{2}-2\left\Vert
\frac{U^{\prime\prime}}{y-y^{i}}\right\Vert _{L^{p_{1}}}\left\Vert
u\right\Vert _{L^{p_{2}}}^{2}\\
&  \geq\left\Vert u^{\prime}\right\Vert _{L^{2}}^{2}-C\left\Vert
U^{\prime\prime}\right\Vert _{H^{s-1}}\left\Vert u^{\prime}\right\Vert
_{L^{2}}^{2}\\
&  \geq\left(  1-C\varepsilon_{0}\right)  \left\Vert u^{\prime}\right\Vert
_{L^{2}\left(  -1,1\right)  }^{2}>0\text{.}%
\end{align*}
This shows that $L_{i}>0$ when $\varepsilon_{0}\ $is sufficiently small and
the proof is completed.
\end{proof}

\section{Linear decay problem}

In this Section, we studied the linearized Euler equation around Couette flow.
In the vorticity form, the linearized equation becomes
\begin{equation}
\omega_{t}+y\omega_{x}=0, \label{linearized-vorticity}%
\end{equation}
where $\omega\left(  t,x,y\right)  $ has $x-$period $T$. If the initial
vorticity $\omega\left(  t=0\right)  =\omega^{0}\left(  x,y\right)  $, then
\begin{equation}
\omega\left(  t,x,y\right)  =\omega^{0}\left(  x-ty,y\right)  .
\label{formula-linear-vort}%
\end{equation}
Notice that any $\omega=\omega\left(  y\right)  \ $is a steady solution of
(\ref{linearized-vorticity}). For a general solution
(\ref{linearized-vorticity})$,$ the $x-$independent component of $\omega
\ $remains steady and does not affect the evolution of the vertical velocity
$v\left(  t\right)  $. So we only consider $\omega^{0}\left(  x,y\right)  $
with $\int_{0}^{T}\omega^{0}\left(  x,y\right)  dx=0$, and for such functions
the Fourier series representation is
\[
\omega^{0}\left(  x,y\right)  =\sum_{0\neq k\in\mathbf{Z}}e^{i\frac{2\pi}%
{T}kx}\omega_{k}^{0}\left(  y\right)  .
\]
%CCZ
Under this assumption, it is easy to see that such a vorticity field uniquely
determines a velocity field satisfying
\begin{equation}
\int_{0}^{T}\vec{u}(x,y)dx\equiv0. \label{avergae}%
\end{equation}
To simplify notations, we take $T=2\pi$ below. We define the space
$H_{x}^{s_{x}}H_{y}^{s_{y}}$ by
\[
h=\sum_{0\neq k\in\mathbf{Z}}e^{ikx}h_{k}\left(  y\right)  \in H_{x}^{s_{x}%
}H_{y}^{s_{y}}\text{ iff }\left\Vert h\right\Vert _{H_{x}^{s_{x}}H_{y}^{s_{y}%
}}=\left(  \sum_{k\neq0}\left\vert k\right\vert ^{2s_{x}}\left\Vert
h_{k}\right\Vert _{H_{y}^{s_{y}}}^{2}\right)  ^{\frac{1}{2}}<\infty.
\]

\begin{theorem}
\label{thm-linear}Assume $\int_{0}^{T}\omega^{0}\left(  x,y\right)  dx=0$. Let
$\omega\left(  t,x,y\right)  $ be the solution of (\ref{linearized-vorticity})
with $\omega\left(  t=0\right)  =\omega^{0}\left(  x,y\right)  $, and
\[
\vec{u}\left(  t,x,y\right)  =\left(  u\left(  t,x,y\right)  ,v\left(
t,x,y\right)  \right)
\]
is the corresponding velocity satisfying \eqref{avergae}.

(i) If $\omega^{0}\left(  x,y\right)  \in L_{x,y}^{2}$, then
\[
\left\Vert \vec{u}\left(  t,x,y\right)  \right\Vert _{L_{x,y}^{2}}%
\rightarrow0,\ \text{when\ }t\rightarrow\infty.
\]

(ii)If $\omega^{0}\left(  x,y\right)  \in H_{x}^{-1}H_{y}^{1}$, then
\[
\left\Vert \vec{u}\left(  t,x,y\right)  \right\Vert _{L_{x,y}^{2}}=O\left(
\frac{1}{t}\right)  ,\ \text{when\ }t\rightarrow\infty.
\]
\newline

(iii) If $\omega^{0}\left(  x,y\right)  \in H_{x}^{-1}H_{y}^{2},$then
\[
\left\Vert v\left(  t,x,y\right)  \right\Vert _{L_{x,y}^{2}}=O\left(  \frac
{1}{t^{2}}\right)  ,\ \text{when\ }t\rightarrow\infty.
\]

(iv) If $\omega^{0}\left(  x,y\right)  \in H_{x}^{-s}H_{y}^{s}\ \left(
0<s<1\right)  $, then
\[
\left\Vert \vec{u}\left(  t,x,y\right)  \right\Vert _{L_{x,y}^{2}}=o\left(
\frac{1}{t^{s}}\right)  ,\ \text{when }t\rightarrow\infty.\newline%
\]

(v) If $\omega^{0}\left(  x,y\right)  \in H_{x}^{-1}H_{y}^{s}$ $\left(  1\leq
s\leq2\right)  $, then
\[
\left\Vert v\left(  t,x,y\right)  \right\Vert _{L_{x,y}^{2}}=O\left(  \frac
{1}{t^{1+s}}\right)  ,\ \text{when }t\rightarrow\infty.\newline%
\]

\end{theorem}

\begin{proof}
Proof of (i): We shall show that $\omega\left(  t,x,y\right)  \rightarrow0$
weakly in $L_{x,y}^{2}$, then $\left\Vert \vec{u}\left(  t,x,y\right)
\right\Vert _{L_{x,y}^{2}}\rightarrow0$ because of the compactness of the
mapping $\omega\rightarrow\vec{u}$ in $L_{x,y}^{2}$. To show the weak
convergence, we take any test function
\[
\phi\left(  x,y\right)  =\sum e^{ikx}\phi_{k}\left(  y\right)  \in L^{2}.
\]
Then
\begin{align*}
&  \int_{0}^{2\pi}\int_{-1}^{1}\omega\left(  t,x,y\right)  \phi\left(
x,y\right)  \ dydx\\
&  =\int_{0}^{2\pi}\int_{-1}^{1}\omega^{0}\left(  x,y\right)  \phi\left(
x+ty,y\right)  dydx\text{ (by (\ref{formula-linear-vort}))}\\
&  =\sum_{k\neq0}\int_{-1}^{1}\omega_{k}^{0}\left(  y\right)  \phi_{-k}\left(
y\right)  e^{-itky}dy\\
&  =\sum_{\left\vert k\right\vert \leq N}\int_{-1}^{1}\omega_{k}^{0}\left(
y\right)  \phi_{-k}\left(  y\right)  e^{-itky}dy+\sum_{\left\vert k\right\vert
>N}\int_{-1}^{1}\omega_{k}^{0}\left(  y\right)  \phi_{-k}\left(  y\right)
e^{-itky}dy=I+II.
\end{align*}
For any $\varepsilon>0$, we fixed $N$ large enough such that
\[
\left\vert II\right\vert \leq\left(  \sum_{\left\vert k\right\vert
>N}\left\Vert \omega_{k}^{0}\left(  y\right)  \right\Vert _{L_{y}^{2}}%
^{2}\right)  ^{\frac{1}{2}}\left\Vert \phi\right\Vert _{L_{x,y}^{2}%
}<\varepsilon\text{.}%
\]
By Riemann-Lesbegue Theorem, $\left\vert I\right\vert \rightarrow0$ when
$t\rightarrow\infty.$ Since $\varepsilon$ is arbitrary, this proves that
\[
\int_{0}^{2\pi}\int_{-1}^{1}\omega\left(  t,x,y\right)  \phi\left(
x,y\right)  \ dydx\rightarrow0\text{ when }t\rightarrow\infty.
\]

Proof of (ii): Define the space $\tilde{H}^{1}\ $for the stream function by
\[
\tilde{H}^{1}=\left\{  \psi\in H_{x,y}^{1}\ |\ \psi= 0 \text{ on }\left\{
y=\pm1\right\}  ,\psi\text{ is }T\text{-periodic in }x\right\}  .
\]
Then
\[
\psi=\sum e^{ikx}\psi_{k}\left(  y\right)  \in H_{x,y}^{1}%
\]
implies that $\psi_{k}\left(  y\right)  \in H_{0}^{1}\left(  -1,1\right)  $
and $\sum_{k}\left\Vert \psi_{k}\right\Vert _{H_{y}^{1}}^{2}<\infty$. By a
duality lemma in \cite{lin-imrn04},
\begin{align*}
\left\Vert \vec{u}\left(  t,x,y\right)  \right\Vert _{L_{x,y}^{2}}  &  \leq
C\sup_{\psi\in\tilde{H}^{1},\left\Vert \psi\right\Vert _{H^{1}}\leq
1}\left\vert \int_{0}^{2\pi}\int_{-1}^{1}\omega\left(  t,x,y\right)
\psi\left(  x,y\right)  \ dydx\right\vert \\
&  =C\sup_{\psi\in\tilde{H}^{1},\left\Vert \psi\right\Vert _{H^{1}}\leq1}%
\sum_{k\neq0}\left\vert \int_{-1}^{1}\omega_{k}^{0}\left(  y\right)  \psi
_{-k}\left(  y\right)  e^{-itky}dy\right\vert \\
&  =\frac{C}{t}\sup_{\psi}\sum_{k\neq0}\frac{1}{\left\vert k\right\vert }%
\int_{-1}^{1}\frac{d}{dy}\left(  \omega_{k}^{0}\left(  y\right)  \psi
_{-k}\left(  y\right)  \right)  e^{-itky}dy\\
&  \leq\frac{C}{t}\sup_{\psi}\left(  \sum_{k\neq0}\frac{1}{\left\vert
k\right\vert ^{2}}\left\Vert \omega_{k}^{0}\left(  y\right)  \right\Vert
_{H_{y}^{1}}^{2}\right)  ^{\frac{1}{2}}\left(  \sum_{k}\left\Vert \psi
_{k}\right\Vert _{H_{y}^{1}}^{2}\right)  ^{\frac{1}{2}}\\
&  \leq\frac{C}{t}\left\Vert \omega^{0}\right\Vert _{H_{x}^{-1}H_{y}^{1}%
}\text{.}%
\end{align*}

Proof of (iii): Note that
\[
-\Delta v=\omega_{x}=\omega_{x}^{0}\left(  x-ty,y\right)  \text{ in }%
\Omega=\left(  0,2\pi\right)  \times\left(  -1,1\right)  ,
\]
and $v=0$ on $\left\{  y=\pm1\right\}  $. Define the function $\varphi\left(
t,x,y\right)  $ by solving $-\Delta\varphi=v\ $in $\Omega$ and $\varphi
=0\ $on\ $\left\{  y=\pm1\right\}  .$ Let
\[
v\left(  t,x,y\right)  =\sum_{k\neq0}e^{ikx}v_{k}\left(  y,t\right)  \text{
and }\varphi\left(  t,x,y\right)  =\sum_{k\neq0}e^{ikx}\varphi_{k}\left(
y,t\right)  ,
\]
where$\ \varphi_{k}$ satisfies that
\[
\left(  -\frac{d^{2}}{dy^{2}}+k^{2}\right)  \varphi_{k}=v_{k},\ \varphi
_{k}\left(  \pm1\right)  =0.
\]
Then
\begin{align*}
\left\Vert v\right\Vert _{L^{2}\left(  \Omega\right)  }^{2}  &  =\int
\int_{\Omega}\bar{\varphi}\left(  t,x,y\right)  \omega_{x}^{0}\left(
x-ty,y\right)  dxdy\\
&  =\sum_{k\neq0}\int_{-1}^{1}ik\bar{\varphi}_{k}\left(  y,t\right)
\omega_{k}^{0}\left(  y\right)  e^{-itky}dy\\
&  =-\frac{1}{t}\sum_{k\neq0}\int_{-1}^{1}\frac{d}{dy}\left(  \bar{\varphi
}_{k}\left(  y,t\right)  \omega_{k}^{0}\left(  y\right)  \right)
e^{-itky}dy\\
&  =\frac{1}{it^{2}}\sum_{k\neq0}\frac{1}{k}\left(  e^{-itky}dy\frac{d}%
{dy}\left(  \bar{\varphi}_{k}\left(  y,t\right)  \omega_{k}^{0}\left(
y\right)  \right)  |_{-1}^{1}-\int_{-1}^{1}\frac{d^{2}}{dy^{2}}\left(
\bar{\varphi}_{k}\left(  y,t\right)  \omega_{k}^{0}\left(  y\right)  \right)
e^{-itky}dy\right) \\
&  \leq\frac{C}{t^{2}}\sum_{k\neq0}\frac{1}{\left\vert k\right\vert
}\left\Vert \bar{\varphi}_{k}\left(  y,t\right)  \omega_{k}^{0}\left(
y\right)  \right\Vert _{H^{2}\left(  -1,1\right)  }\leq\frac{C}{t^{2}}%
\sum_{k\neq0}\frac{1}{\left\vert k\right\vert }\left\Vert \varphi_{k}\left(
y,t\right)  \right\Vert _{H_{y}^{2}}\left\Vert \omega_{k}^{0}\left(  y\right)
\right\Vert _{H_{y}^{2}}\\
&  \leq\frac{C}{t^{2}}\sum_{k\neq0}\frac{1}{\left\vert k\right\vert
}\left\Vert v_{k}\left(  y,t\right)  \right\Vert _{L_{y}^{2}}\left\Vert
\omega_{k}^{0}\left(  y\right)  \right\Vert _{H_{y}^{2}}\\
&  \leq\frac{C}{t^{2}}\left(  \sum_{k\neq0}\frac{1}{\left\vert k\right\vert
^{2}}\left\Vert \omega_{k}^{0}\left(  y\right)  \right\Vert _{H_{y}^{2}}%
^{2}\right)  ^{\frac{1}{2}}\left(  \sum_{k}\left\Vert v_{k}\left(  y,t\right)
\right\Vert _{L_{y}^{2}}^{2}\right)  ^{\frac{1}{2}}\\
&  \leq\frac{C}{t^{2}}\left\Vert \omega^{0}\right\Vert _{H_{x}^{-1}H_{y}^{2}%
}\left\Vert v\right\Vert _{L^{2}\left(  \Omega\right)  },
\end{align*}
therefore
\[
\left\Vert v\right\Vert _{L^{2}\left(  \Omega\right)  }\leq\frac{C}{t^{2}%
}\left\Vert \omega^{0}\right\Vert _{H_{x}^{-1}H_{y}^{2}}.
\]
The decay rates in (iv) and (v) follow from (i)-(iii) by interpolation. This
finishes the proof of Theorem \ref{thm-linear}.
\end{proof}

\begin{remark}
\label{rmk-decay rate}The decay rates $O\left(  1/t\right)  \ $for $\left\Vert
u\right\Vert _{L^{2}}$ and $O\left(  1/t^{2}\right)  \ $for $\left\Vert
v\right\Vert _{L^{2}}$ in Theorem \ref{thm-linear} (ii), (iii) are optimal.
They cannot be improved even for smooth initial vorticity. Consider a single
mode solution with $\omega^{0}\left(  x,y\right)  =e^{ikx}\phi\left(
y\right)  $ and $\phi\left(  y\right)  \in C^{\infty}\left(  -1,1\right)  $.
Then $\omega\left(  t,x,y\right)  =e^{ikx}e^{-ikty}\phi\left(  y\right)  $ and
by Poisson's equation the stream function is $\psi\left(  t,x,y\right)
=e^{ikx}\psi_{k}\left(  t,y\right)  ,$where $\psi_{k}\left(  t,y\right)  $
satisfies
\[
\left(  -\frac{d^{2}}{dy^{2}}+k^{2}\right)  \psi_{k}\left(  t,y\right)
=e^{-ikty}\phi\left(  y\right)  ,\ \psi_{k}\left(  t,\pm1\right)  =0\text{.}%
\]
Denote $G\left(  y,y_{0}\right)  $ to be the Green's function given by%
\[
G\left(  y,y_{0}\right)  =\frac{1}{k\sinh k}\sinh k\left(  y_{<}+1\right)
\sinh k\left(  1-y_{>}\right)  ,
\]
where $y_{<}$ and $y_{>}$ are the lesser and greater of $y$ and $y_{0}$
respectively. Then we have
\begin{equation}
\psi_{k}\left(  t,y\right)  =\int_{-1}^{1}G\left(  y,y_{0}\right)
e^{-ikty_{0}}\phi\left(  y_{0}\right)  dy_{0}, \label{formula-psi-k}%
\end{equation}
and the $1/t^{2}$ decay of $\psi_{k}\left(  t,y\right)  $ follows from
integration by parts because $G\left(  y,y_{0}\right)  $ is $C^{1}$ and its
derivative is piecewise differentiable. Moreover, by explicit evaluation of
the integral in (\ref{formula-psi-k}), it can be shown that
\[
\psi_{k}\left(  t,y\right)  =\frac{1}{t^{2}}f_{k}\left(  y\right)
e^{-ikty}+O\left(  \frac{1}{t^{3}}\right)
\]
where $f_{k}\left(  y\right)  $ is not identically zero. Thus
\[
\left\Vert u\left(  t,x,y\right)  \right\Vert _{L_{x,y}^{2}}=\left\Vert
\psi_{k}^{\prime}\right\Vert _{L_{y}^{2}}\sim\frac{1}{t}%
\]
and
\[
\text{ }\left\Vert v\left(  t,x,y\right)  \right\Vert _{L_{x,y}^{2}%
}=k\left\Vert \psi_{k}\right\Vert _{L_{y}^{2}}\sim\frac{1}{t^{2}}\text{. }%
\]
The same decay rate $O\left(  \frac{1}{t^{2}}\right)  $ for $v\left(
t,x,y\right)  $ was obtained in (\cite{bouchet09}, \cite{brown-stewartson-80},
\cite{orr-1907}). Our main purpose in this section is to get the linear decay
for most general perturbations. We note that the calculations in
\cite{case-1960} contain mistakes and only yield the estimate
\[
\psi_{k}\left(  t,y\right)  =\frac{1}{t}g_{k}\left(  y\right)  e^{-ikty}%
+O\left(  \frac{1}{t^{2}}\right)  ,
\]
from which only $O\left(  \frac{1}{t}\right)  $ decay is obtained for
$v\left(  t,x,y\right)  \ $and no decay is obtained for$\ u\left(
t,x,y\right)  =\psi_{k}^{\prime}\left(  t,y\right)  $.
\end{remark}

\begin{center}
\bigskip

{\Large Acknowledgement}
\end{center}

This work is supported partly by the NSF grants DMS-0908175 (Lin) and
DMS-0801319 (Zeng).

\end{document}